\documentclass[12pt,reqno]{amsart}
\usepackage{amsthm,amsmath,amssymb,amsfonts, amscd,graphics, latexsym,
enumerate, stmaryrd,xspace,verbatim,epic,eepic, dsfont}

\usepackage{hyperref}
\usepackage[all]{xy}

\usepackage[active]{srcltx}

\newtheorem{theorem}{Theorem}[section]
\newtheorem{proposition}[theorem]{Proposition}
\newtheorem{corollary}[theorem]{Corollary}
\newtheorem{lemma}[theorem]{Lemma}

\theoremstyle{definition}
\newtheorem{definition}[theorem]{Definition}

\newtheorem{example}[theorem]{Example}

\newtheorem{remark}[theorem]{Remark}

\theoremstyle{remark}

\DeclareMathOperator{\B}{B}
\DeclareMathOperator{\spec}{Spec}
\DeclareMathOperator{\Spec}{\mathbf{Spec}} 
\DeclareMathOperator{\Vect}{Vect}
\DeclareMathOperator{\Aut}{Aut}
\DeclareMathOperator{\BAut}{\mathbf{Aut}}

\DeclareMathOperator{\Ob}{Ob}
\DeclareMathOperator{\Par}{Par}
\DeclareMathOperator{\EFPar}{EFPar}
\DeclareMathOperator{\Hom}{Hom}
\DeclareMathOperator{\BHom}{\mathbf{Hom}}
\DeclareMathOperator{\BIsom}{\mathbf{Isom}}
\DeclareMathOperator{\Sym}{\mathbf{Sym}} 
\DeclareMathOperator{\id}{id}

\DeclareMathOperator{\Rep}{Rep}
 
\DeclareMathOperator{\Ind}{Ind} 
\DeclareMathOperator{\Z}{Z} 
\DeclareMathOperator{\pr}{pr}
\DeclareMathOperator{\Gal}{Gal}
\DeclareMathOperator{\h}{H} 
\DeclareMathOperator{\R}{R} 

\newcommand{\thickslash}{\mathbin{\!\!\fatslash}}

\begin{document}

\title[Tamely ramified torsors and parabolic bundles]{Tamely ramified torsors and 
parabolic bundles}

\author[I. Biswas]{Indranil Biswas}

\address{School of Mathematics, Tata Institute of Fundamental
Research, Homi Bhabha Road, Mumbai 400005, India}

\email{indranil@math.tifr.res.in}

\author[N. Borne]{Niels Borne}

\address{Universit\'e Lille 1, Cit\'e scientifique
U.M.R. CNRS 8524, U.F.R. de Math\'ema\-tiques
59 655 Villeneuve d'Ascq C\'edex, France}

\email{Niels.Borne@math.univ-lille1.fr}



\begin{abstract}

Given a variety $X$, a normal crossings divisor $D\subset X$, and a finite abelian group scheme 
$G$, we relate, in the case of abelian monodromy, the following two:
\begin{itemize}
\item existence of a $G$-torsor with prescribed ramification, and

\item existence of essentially finite parabolic vector bundles with prescribed weights.
\end{itemize}
\end{abstract}

\maketitle

\section{Introduction}

\subsection{Tamely ramified covers and parabolic bundles}
\label{sub:tamely_ramified_covers_and_parabolic_bundles}

A classical construction, due to A. Weil, associates to a Galois \'etale cover $Y\longrightarrow X$ 
of algebraic varieties over a field $k$, and a representation $V$ of $G\,=\,\Gal(Y/X)$, 
a vector bundle $\mathcal E_V$ on $X$. These bundles are finite, in the sense that 
each of these admits a non-trivial tensor relation \cite{We}. In fact, the family
$\mathcal E_V$, 
when $V$ varies in all representations, allows us to reconstruct the covering $Y\longrightarrow X$.

This was later extended by M. Nori to $G$-torsors, where $G$ is a finite group 
scheme over $k$. In modern terms, the data of such a torsor is equivalent to the 
data of a tensor exact functor $\Rep_k(G) \longrightarrow \Vect(X)$.

Another way to generalize Weil's construction is to consider a Galois cover $Y\longrightarrow X$ 
which is tamely ramified along a normal crossings divisor $D$, in the sense of 
Grothendieck--Murre. Then it is necessary to endow vector bundles with a parabolic 
structure and the upshot is that such a covering defines a tensor functor $\Rep_k(G) 
\longrightarrow \Par(X,D)$ from representations of the Galois group to vector bundles with a 
parabolic structure along $D$ (see \cite[Proposition 3.2.2 and Th\'eor\`eme 
2.4.7]{bor:rep}).

\subsection{Tamely ramified torsors}
\label{sub:tamely_ramified_torsors}

The aim of this article is to define a notion of ``tamely ramified $G$-torsor'' with 
encapsulates both situations, and, mostly, to give a criterion for the existence of such 
a torsor, when the ramification data is prescribed.

Our very natural definition is inspired by the characterization of tamely ramified 
cover deduced from Abhyankar's lemma : a tamely ramified torsor is a $G$-invariant 
cover that is \emph{fppf} locally on the base induced by a Kummer cover, along a 
group monomorphism $\mathbb \mu_r\longrightarrow G$. This is consistent with the philosophy and 
the general definition given in \cite{CEPT:tame}. Setting $G\,=\,\mathbb \mu_r$, one 
recovers classical objects for instance uniform cyclic covers in 
\cite{arsvist:uniform}; the case of
$G\,=\,\mathbb \mu_r$ was extensively used as the ``cyclic cover trick'' (see for 
instance \cite{esnvieh:vanishing} for an application to Hodge theory).

\subsection{Main result}
\label{sub:main_result}

Our principal result is a partial answer to the following question : If we assign to each 
irreducible component $D_i$ of $D$ a multiplicity $r_i$, does there exists a tamely 
ramified $G$ cover $Y\longrightarrow X$ with ramification index $r_i$ at $D_i$ ? It is easy to 
work out a necessary condition : if such a cover exists, it can be used to 
produce essentially finite parabolic vector bundles with weights any rational number 
of denominator $r_i$. Our main result, Theorem \ref{crit_exist}, states essentially 
that, when $G$ is abelian, this necessary condition is also sufficient. More precisely, if any 
admissible weight occurs as the weight of an essentially finite parabolic vector 
bundle with abelian monodromy, then there exists a ramified $G$-torsor with the 
exact prescribed ramification.

The main point of the proof consists of introducing a certain algebraic stack called 
stack of roots, build out of the ramification data. The principal difficulty, worked 
out in \S~\ref{sec:stack_int}, is to relate ramified $G$-torsors and (usual) 
$G$-torsors on the stack of roots. Although it looks obvious from the definitions,  this is by no means a trivial issue : as David Rydh informed us, the key Proposition \ref{prop:alper_crit_torsors_local} turns out to be false if we don't assume $G$ to be abelian. We include his counter-examples in Appendix \ref{sec:David_Rydh_counter_example}. So this hypothesis in Theorem \ref{crit_exist} seems sharp. This restriction on the structure group also appears in the closely related work 
\cite{marques:slices}.

Once this is done, the question of the existence of a tamely ramified cover translates 
into the question of Nori-uniformization of the stack of roots; the latter can be solved 
thanks to the Tannakian characterization given in our previous work 
\cite{bisbor:fundamental_gerbe}. In the background, we also rely on the 
interpretation of parabolic vector bundles as ordinary vector bundles on the stack 
of roots (\cite{bis:par,bor:corr}).

While considering the issue of Nori-uniformization, we are naturally led to use 
Nori's fundamental group, or rather Nori's fundamental gerbe (\cite{bv:nori_gerbe}). 
In the Appendix \ref{sec:abelian_objects_in_tannakian_categories}, we give the 
following Tannakian interpretation of gerbes with abelian monodromy, which might be 
of independent interest : Each object has a Tannakian monodromy coming from the base 
-- we call these objects as \emph{basic}.

\subsection{Notations and conventions}

We work over an arbitrary field $k$ and denote by $S=\spec k$ the base scheme. Fix a 
locally noetherian scheme $X/S$ and a finite family $\mathbf D \,=\, (D_i)_{i\in I}$ of 
reduced, irreducible, distinct Cartier divisors on $X$. We add to our data a family 
${\mathbf r} \,=\, (r_i)_{i\in I}$ of positive integers. We say that $D=\bigcup_{i\in I} 
D_i$ is a simple (or strictly) normal crossings divisor $D$ on $X$ if $X$ is regular 
at each point of $D$ and $D$ is locally defined by a subsystem of a regular system 
of parameters (see \cite[Definition 1.8.2]{gm:tame}). We don't assume this is the 
case from the start.

If $\mathcal X$ is an algebraic stack, we will denote by $\left|\mathcal X\right|$ the associated topological space and by $\left|\mathcal X\right|_0$ the set of closed points.

\section{Torsors with prescribed ramification}

\subsection{Tamely ramified torsors}

Given a closed point $x$ of $X$, we denote by $I_x$ the set of $i$ in $I$ such that $x$ belongs to the support of $D_i$. We also set $(\mathbf r_x)_i$ as $r_i$ if $i\in I_x$ and $1$ otherwise; this defines a multi-index $\mathbf r_x$.

\begin{definition}\label{def:Kum-com}
Let $x$ be a closed point of $X$, and let $U\,=\,\spec R$ be an open affine neighbourhood.
For each $i\,\in\, I_x$, let $s_i\,\in\, R$ be a local equation for the Cartier divisor $D_i$.
The corresponding \emph{Kummer morphism} on $U$ with ramification data $(\mathbf D,\mathbf r)$ is the morphism $$Z\,=\,\spec R [\mathbf t] / \left(\mathbf t^{\mathbf r} - \mathbf s\right)\longrightarrow \spec R\, ,$$ where the algebra is the tensor product over $R$ of $R[t_i]/\left(t_i^{r_i}-s_i\right)$ for all $i$ in $I_x$. \end{definition}

Notice that $Z$ is naturally endowed with a $\mu_{\mathbf r_x}$-action corresponding to the obvious $\mathbb Z^{I_x}/\mathbf r_x$-grading, and that $Z\longrightarrow \spec R$ is invariant.

 \begin{definition}
\label{def:tame_torsor}
	 Let $G/k$ be a finite abelian group scheme. A \emph{tamely ramified $G$-torsor} over $X$ with ramification data $(\mathbf D,\mathbf r)$ is the data of a scheme $Y$ endowed with an action of $G$ and a finite and flat $G$-invariant morphism $Y \longrightarrow X$ such that for each closed point $x$ of $X$, there exists 
	 a monomorphism $\mu_{\mathbf r_x} \longrightarrow G$ defined over an extension $k'/k$ such that in a \emph{fppf} neighbourhood of $x$ in $X$, the morphism $Y\longrightarrow X$ is isomorphic to $\left(Z\otimes_k k'\right)\times^{\mu_{\mathbf r_x}}G$, where $Z\longrightarrow \spec R$ is the Kummer morphism with ramification data $(\mathbf D,\mathbf r)$ locally defined by the choice of equations of $\mathbf D$ at $x$.
\end{definition}

 \begin{remark}
	 \label{rem:def-tame-torsor}
	 \mbox{}
\begin{enumerate}
	\item \label{rem:def-tame-torsor-detail} By `a \emph{fppf} neighbourhood of
$x$ in $X$' we mean : there exists a \emph{fppf} cover $X'\longrightarrow \spec \left( R\otimes_k k'\right)$ such that if $Z'=X'\times_{\spec R}Z$, then $ Z'\times^{\mu_{\mathbf r_x}}G$ and $Y'=X'\times_X Y$ are isomorphic as $G$-schemes over $X'$.
	\item 	 When $G\,=\,\mu_r$, one recovers the usual notion of uniform cyclic
cover (see \cite{arsvist:uniform}).
	\item This definition is mainly interesting when $D$ is a simple normal crossings divisor. Indeed, assume $X$ is normal and $Y\longrightarrow X$ is an abelian $G$-Galois cover tamely ramified along a normal crossings divisor $D$, in the classical sense (\cite[D\'efinition 2.2.2]{gm:tame}). It will turn out that $Y\longrightarrow X$ is a tamely ramified torsor under the abelian constant group scheme $G$, in the sense of Definition \ref{def:tame_torsor}, if and only if $D$ is a simple normal crossings divisor. See Remark \ref{rem:tame2stack} for a sketch of a proof and more details.

\end{enumerate}
\end{remark}

\subsection{Parabolic bundles}

We will employ C. Simpson's functorial definition of parabolic vector bundles as follows (see for instance \cite[\S 2]{bor:rep} for full details). We endow the set $\frac{1}{\mathbf r}\mathbb Z^I = \prod_{i\in I} \frac{\mathbb Z}{r_i} $ with the component-wise partial order, and denote by the same letter the associated category. The exponent $\cdot ^{op}$ stands for the opposite category.

\begin{definition}
A \emph{parabolic vector bundle} on $(X, \mathbf D)$ with weights in
$\frac{1}{\mathbf r}\mathbb Z^I$ consists of
\begin{itemize}
\item the data of a functor $\mathcal E_\cdot \,:\, \left(\frac{1}{\mathbf r}
	\mathbb Z^I\right)^{op} \longrightarrow \Vect X$ and,

\item for each integral multi-index $\mathbf l$ in $\mathbf Z^I$, a natural
transformation $\mathcal E_{\cdot+\mathbf l} \simeq 
\mathcal E_{\cdot}\otimes_{\mathcal O_X} \mathcal O_X(-\mathbf l \cdot \mathbf D)$.
\end{itemize}
The two pieces of data 
	must satisfy the following compatibility condition: for $\mathbf l \geq \mathbf 0$, the natural transformation $$\mathcal E_{\cdot+\mathbf l} \simeq \mathcal E_{\cdot}\otimes_{\mathcal O_X} \mathcal O_X(-\mathbf l \cdot \mathbf D)$$ is compatible with the morphisms $\mathcal E_{\cdot+\mathbf l}\longrightarrow \mathcal E_{\cdot}$ and $\mathcal O_X(-\mathbf l \cdot \mathbf D) \longrightarrow \mathcal O_X$.
\end{definition}

One way to analyze parabolic bundles is to look at their local weights, in the following sense. 

\begin{definition}\footnote{We thank Eric Ahlqvist for correcting our original definition.}
Given a closed point $x$ of $D$, and $\mathbf l \in \mathbb Z^{I}$, we
say that a parabolic vector bundle $\mathcal E_\cdot$ on $(X, \mathbf D)$ with weights in $\frac{1}{\mathbf r}\mathbb Z^I$ \emph{admits $\mathbf l /\mathbf r $ as a weight at $x$} if the %monomorphism $\left(\mathcal E_{({\mathbf l}+{\mathbf 1}) /\mathbf r}\right)_x \longrightarrow \left(\mathcal E_{\mathbf l /\mathbf r}\right)_x$ is not an isomorphism.
morphism $\left(\bigoplus_{i\in I_x}  \mathcal E_{({\mathbf l}+{\mathbf e_i}) /\mathbf r}\right)_x \longrightarrow \left(\mathcal E_{\mathbf l /\mathbf r}\right)_x$ is not an epimorphism ($(\mathbf e_i)_{i\in I}$ is the canonical basis of $\mathbb Z^I$). 
\end{definition}

\begin{example}
Any vector bundle $\mathcal E$ on $X$ can be endowed with a trivial parabolic structure denoted by $	\underline{\mathcal E}_\cdot$, whose set of weights is $\mathbb Z^{I}$ at any closed point $x$ of $D$.
\end{example}

We will be mainly concerned by the full sub-category of \emph{essentially finite} parabolic vector bundles. Their definition is as follows : first, parabolic vector bundles (with arbitrary weights) form a tensor category denoted by $\Par(X,\mathbf D)$. This enables to define \emph{finite} parabolic vector bundles as objects of this category satisfying a non trivial tensor relation $f(\mathcal E_\cdot )\simeq g(\mathcal E_\cdot )$, where $f$ and $g$ are two distinct polynomials in $\mathbb N[t]$. Then, just as in \cite{bv:nori_gerbe}, one defines an essentially finite object in $\Par(X,\mathbf D)$ as the kernel of an homomorphism between two finite objects. This notion gives rise to the full sub-category $\EFPar(X,\mathbf D)$ of essentially finite vector bundles. When $X$ is proper, geometrically connected and geometrically reduced, it follows from \cite[Theorem 7.13]{bv:nori_gerbe} that the category $\EFPar(X,\mathbf D)$ is Tannakian\footnote{To check this, one applies this theorem to the stacks of roots
introduced \S~\ref{sub:stack_of_roots}, see Remark \ref{rem:correspondence}.}.

As in any Tannakian category (see Appendix \ref{sec:abelian_objects_in_tannakian_categories}), one can construct a full sub-category by selecting only objects with abelian holonomy gerbe. For convenience of the reader, we spell out the Tannakian formulation explicitly in our situation.

\begin{definition}[]
\label{def:basic_parabolic_sheaf}
Let $\mathcal E_\cdot$ be an essentially finite parabolic sheaf, and denote by
$\pi(\mathcal E_\cdot)$ its fundamental group within $\EFPar(X,\mathbf D)$ (Definition 
\ref{def:fund_group_tannaka}). We will call that $\mathcal E_\cdot$ as 
\emph{basic} if there exists a $k$-affine group scheme $G$ and an isomorphism $\pi(\mathcal E) \simeq G\times_k \underline{\mathcal O_X}_\cdot$ as group schemes within $\EFPar(X,\mathbf D)$.
\end{definition}

When $X$ is proper, geometrically connected and geometrically reduced, we will use the notation $\EFPar(X,\mathbf D)^{bas}$ for the Tannakian category of essentially finite basic parabolic sheaves. This is just the Tannakian sub-category of objects of $\EFPar(X,\mathbf D)$ with abelian holonomy gerbe (see Appendix \ref{sec:abelian_objects_in_tannakian_categories}).

\subsection{Criterion of existence}

This is the main theorem of the paper.

\begin{theorem}\label{crit_exist}
Let $X/k$ be a proper scheme of finite type, endowed with a simple normal crossings divisor $D$. Assume that $X$ is geometrically connected and geometrically reduced. There exists a finite abelian group scheme $G/k$ and a
tamely ramified $G$-torsor $Y\longrightarrow X$ with ramification data $(\mathbf D,\mathbf r)$ if and only if for
each closed point $x$ of $D$, and for all $\mathbf l \in \mathbb Z^I $, such that
$0\leq \mathbf l<\mathbf r_x$, there exists an object $\mathcal E_\cdot$ in 
$\EFPar(X,\mathbf D)^{bas}$ with weights in $\frac{1}{\mathbf r}\mathbb Z^I$,
satisfying the condition that $\mathcal E_\cdot$ admits $\mathbf l /\mathbf r $ as a weight at $x$. 
\end{theorem}

The proof is given in Section \ref{sec:stack_int} after collecting some preliminary results. 
\begin{remark}\mbox{}
\begin{enumerate}
\item 
The ``only if'' part is straightforward, and we can moreover produce an essentially
finite vector bundle $\mathcal E_\cdot$ independent on $x$ and
$\mathbf l$. Namely, it is easy to see that the parabolic bundle associated to
$\mathcal O_Y$ via the equivariant-parabolic correspondence (see 
\cite{bis:par,bor:corr}) verifies the required property. It is thus rather remarkable
that the ``if'' part is true.
\item We have assumed for simplicity that $D$ is a simple normal crossings divisor. But the theorem should hold with the same proof for a general normal crossings divisor, if on one hand one uses the definition of parabolic sheaves given in \cite{bv:par_sheaves}, and on the other hand one modifies Definition \ref{def:tame_torsor} so that Proposition \ref{prop:equiv} holds with the refined stack of roots used in \cite{bv:par_sheaves}\footnote{While this paper was under review, this result has been shown in greater generality by Eric Ahlqvist, see \cite{ahl:building}.}
\item If one removes the hypothesis that $G$ is abelian (and accordingly one considers non-necessarily basic parabolic vector bundles) the statement of the theorem makes sense but is very likely wrong, see Remark \ref{rem:rydh}. The culprit is our Definition \ref{def:tame_torsor} of a tame torsor which appears to be wrong for our purposes in the non-abelian setup. We do not know how to improve it in order to get a correct statement.
\end{enumerate}
\end{remark}

\section{Stacky interpretation of tamely ramified \texorpdfstring{$G$}{G}-torsors}

\label{sec:stack_int}

\subsection{Stack of roots}
\label{sub:stack_of_roots}

The main tool of the proof of Theorem \ref{crit_exist} is the stack of roots associated to the data $X$, 
$(\mathbf D,\mathbf r)$. For completeness, we recall A. Vistoli's definition.

\begin{definition}[\cite{agv:gw}, Appendix B]
	The \emph{stack of roots} $\sqrt[\mathbf r]{\mathbf D/X}$ is the stack classifying $\mathbf r$-th roots of $(\mathcal O_X(\mathbf D), s_{\mathbf D})$.
\end{definition}

Here $s_{\mathbf D}$ denotes the family of canonical sections. So on $\sqrt[\mathbf 
r]{\mathbf D/X}$ there is, for each $i\in I$, a universal $r_i$-th root 
$(\mathcal N_i,t_i)$ of $(\mathcal O_X(D_i),s_{D_i})$. For each $\mathbf l \,\in\, 
\mathbb Z^I$, we denote $\otimes_{i \in I} 
{\mathcal N_i}^{\otimes l_i}$ by $\mathcal N^{\otimes \mathbf l}$. For convenience, we will denote $\sqrt[\mathbf r]{\mathbf D/X}$ by $\mathfrak X$ in the rest of the proof. We will also use the notation $\pi: \mathfrak X \longrightarrow X$ for the morphism to the moduli space. For each $i\in I$, there is a unique effective Cartier divisor $\mathcal D_i$ on $\mathfrak X$ such that $(\mathcal O_{\mathfrak X}(\mathcal D_i),s_{\mathcal D_i})\simeq (\mathcal N_i,t_i)$. We put $\mathcal D = \cup_{i\in I} \mathcal D_i$.

The stacks of roots are locally quotient stacks of Kummer morphisms :

\begin{lemma}
\label{lem:kum-quot}
Let $U=\spec R\subset X$ an affine open subset and for each $i\in I$, let $s_i\in R$ a non-zero element giving a local equation of $D_i$. Then there exists a canonical $R$-isomorphism~:

\[ U\times _X \sqrt[\mathbf r]{\mathbf D/X} \,\simeq\,\left [ \left(\spec R [\mathbf t] / \left(\mathbf t^{\mathbf r} - \mathbf s\right)\right) / \mu_\mathbf r \right]\] 
where the algebra is the tensor product over $R$ of $R[t_i]/\left(t_i^{r_i}-s_i\right)$ for all $i$ in $I$.
\end{lemma}

\begin{proof}
See \cite[Corollaire 3.5]{bor:corr} or \cite[Corollaire 2.4.5]{bor:rep}.
\end{proof}

\begin{remark}
\label{rem:kum-quot}
If $x\in U$, by shrinking $U$, we can assume that $x$ belongs to all branches
meeting $U$, and then, with the notations of Definition \ref{def:Kum-com}, we get
that $ U\times _X \sqrt[\mathbf r]{\mathbf D/X} \simeq [Z/\mu_{\mathbf r_x}]$.
\end{remark}

\begin{lemma}
\label{lem:residual-gerbe}
Let $x\,\in\, \sqrt[\mathbf r]{\mathbf D/X}$ be a closed point. Then the residual gerbe $\mathcal G_x$ at $x$ exists, is neutral, and (non canonically) isomorphic to $\B_{k(x)} \mu_{\mathbf r_x}$, where $k(x)$ is the field of definition of the image of $x$ in $X$.
\end{lemma}

\begin{proof}
	We use \cite[Tag 06ML]{stacks-project} as a reference for residual gerbes. From Lemma \ref{lem:kum-quot} and Remark \ref{rem:kum-quot}, we see that $x$ belong to a closed subscheme of $\sqrt[\mathbf r]{\mathbf D/X}$ of the form $\left [ \left(\spec k(x) [\mathbf t] / \left(\mathbf t^{\mathbf r_x} \right)\right) / \mu_{\mathbf r_x} \right]$, and \cite[Tag 06MT]{stacks-project} enables to conclude.
\end{proof}

\subsection{The equivalence}

This is the relation between the stack of roots and tamely ramified torsors.

\begin{proposition}\label{prop:equiv}
Let $Y$ be a $k$-scheme endowed with an action of a finite abelian $k$-group scheme $G$ and let $Y\longrightarrow X$ be a finite, flat, and $G$-invariant morphism. Then $Y\longrightarrow X$ factors through a $G$-torsor $Y\longrightarrow \mathfrak X$ if and only if $Y\longrightarrow X$ is a tamely ramified $G$-torsor with ramification data $(\mathbf D,\mathbf r)$.
\end{proposition}

The proof of Proposition \ref{prop:equiv} will be given in the next two sections.

\begin{remark}
\label{rem:equiv}\mbox{}
\begin{enumerate}
\item 
One deduces from the parabolic--orbifold correspondence (see \cite{bis:par,bor:corr}) that $G$-torsors on $\mathfrak X$ can be described as tensor exact functors from representations of $G$ to parabolic bundles. It follows from Proposition \ref{prop:equiv} that it also gives another interpretation of tamely ramified $G$-torsors. This point of view in fact precedes the appearance of stacks of roots, see for instance \cite{bbn:par}.

\item \label{rem:equiv-alg-space} If $Y$ is a $k$-algebraic space, and $Y\longrightarrow
\mathfrak X$ is a $G$-torsor, where $G$ is a finite $k$-group scheme, $Y$ must be a
$k$-scheme (since then $Y$ is finite, hence affine, over $X$, a $k$-scheme by
assumption).
\end{enumerate}
\end{remark}

\subsection{From tamely ramified torsors to torsors on the stack of roots}

We fix a tamely ramified $G$-torsor $Y\longrightarrow X$ throughout this section, where $G$ is a finite, non necessarily abelian, $k$-group scheme. The ramification data is fixed as well, which defines a stack of roots $\mathfrak 
X$. 

\begin{proposition}
\label{prop:hom}
Let $\BHom^G_X(Y,\mathfrak X)$ be the fppf stack over $X$ classifying $X$-morphisms $Y\longrightarrow \mathfrak X$ that are $G$-torsors. Then $\BHom^G_X(Y,\mathfrak X)\longrightarrow X$ is an isomorphism.
\end{proposition}

The proof of Proposition \ref{prop:hom} will be given 
after two auxiliary lemmas.
We deduce first the following straight-forward consequence of it.

\begin{corollary}
\label{cor:tame2stack}
There exists a $X$-morphism $Y\longrightarrow \mathfrak X$ that is a $G$-torsor, and it is unique up to unique isomorphism.
\end{corollary}

\begin{remark}
\label{rem:tame2stack}\mbox{}
\begin{enumerate}
\item It follows that $[Y/G]\simeq \mathfrak X$. 
\item On the other hand, start from $X$, a normal scheme, and $Y\longrightarrow X$ a $G$-Galois cover, tamely ramified along a normal crossings divisor $D$, in the sense of \cite[D\'efinition 2.2.2]{gm:tame}). Then Abhyankar's lemma (\cite[Theorem 2.3.2]{gm:tame}) shows that the quotient stack is Olsson's refined stack of roots (in the sense of \cite[Definition 4.12]{bv:par_sheaves}), which coincides with the usual stack of roots if and only if the branch divisor $D$ is simple normal crossings. It follows that if $Y\longrightarrow X$ is a tamely ramified torsor under the constant group scheme $G$, in the sense of Definition \ref{def:tame_torsor}, then $D$ is a simple normal crossings divisor. The converse is true at least if $G$ is abelian, as follows from \cite[Lemme 3.3.1]{bor:rep} and Proposition \ref{prop:equiv}.
\end{enumerate}
\end{remark}

Here is the first result we will need to prove Proposition \ref{prop:hom}.

\begin{lemma}
	\label{lem:quot}
	Let $Z\,\longrightarrow\, Y$ be a scheme morphism equivariant with respect to a group 
	monomorphism $\psi : H\longrightarrow G$. 
	Then the canonical diagram
	\[
	\xymatrix{
	G\times^H Z \ar[r] \ar[d] & Y \ar[d]\\
	[Z/H] \ar[r] & [Y/G]
	}
\]
is Cartesian.
\end{lemma}

\begin{proof}
Since the diagonal action of $H$ on $G\times Z$ is free, the left vertical arrow is a $G$-torsor, and the result follows.
\end{proof}

We will also need the following rigidity lemma.

\begin{lemma}
\label{lem:rigidity}
Let $\phi$ be an automorphism of $\mathfrak X$ above $X$. There exists a unique isomorphism $\phi\simeq \id$.
\end{lemma}

\begin{proof}
	The data of the automorphism $\phi$ amounts to the data, for each $i\in I$, of another universal $r_i$-th root 
$(\mathcal N'_i,t'_i)$ of $(\mathcal O_X(D_i),s_{D_i})$. An isomorphism from $\phi$ to $\id$ is an isomorphism of $r_i$-th roots from
$(\mathcal N'_i,t'_i)$ to $(\mathcal N_i,t_i)$, for each $i\in I$.

Denote by $\mathfrak U\,= \,\mathfrak X\backslash \mathcal D$ the complement of
$\mathcal D$, and denote by $j\,:\, \mathfrak U\,\longrightarrow\, X$ the corresponding
open immersion. Since $\mathcal D$ is a Cartier divisor, the morphism
$\mathcal O_{\mathfrak X}\longrightarrow j_* \mathcal O_{\mathfrak U}$ is a
monomorphism. Because each $t_i\,:\,\mathcal O_{\mathfrak X}\longrightarrow \mathcal N_i$ (respectively $t'_i:\mathcal O_{\mathfrak X}\longrightarrow \mathcal N'_i$) is an isomorphism when restricted to $\mathfrak U$, uniqueness of isomorphism $\phi\simeq \id$ follows.

To show the existence, we note that since $\phi_{|\mathfrak U}\,=\,\id$, the morphism $\phi$ induces an automorphism $\psi$ of the reduced stack $\mathcal D$. Since $\mathcal O_{\mathfrak X}(-r_i \mathcal D_i)\simeq \pi^*\mathcal O_{ X}(- D_i)$, this automorphism preserve topologically the irreducible components, which means that $\phi^{-1}( \mathcal D_i)=\mathcal D_i$ as (reduced) closed substacks.
Thus we get an isomorphism $\phi^*\mathcal O_{\mathfrak X}(\mathcal D_i)\simeq \mathcal O_{\mathfrak X}(\mathcal D_i)$ preserving the canonical sections.
As $(\phi^*\mathcal O_{\mathfrak X}(\mathcal D_i),\phi^* s_{\mathcal D_i})\simeq (\mathcal N'_i,t'_i)$ and $(\mathcal O_{\mathfrak X}(\mathcal D_i), s_{\mathcal D_i})\simeq (\mathcal N_i,t_i)$, this concludes the proof.
\end{proof}

\begin{remark}
\label{rem:rigidity}
In fact, uniqueness is not needed in the proof of Proposition \ref{prop:hom}.
\end{remark}

\begin{proof}[Proof of Proposition \ref{prop:hom}]
We first show that $\BHom^G_X(Y,\mathfrak X)\longrightarrow X$ is an epimorphism. 
Let us thus prove that this morphism has sections in a fppf neighbourhood of $x\,\in \,
X$. We use notations of Definition \ref{def:tame_torsor} and Remark \ref{rem:def-tame-torsor} \eqref{rem:def-tame-torsor-detail}. In particular by Lemma \ref{lem:quot}, we get an $X'$-isomorphism $[Y'/G]\simeq [Z'/\mu_{\mathbf r_x}]$. 
We can assume that $U$ meets only components containing $x$. Then Remark \ref{rem:kum-quot} gives and $X'$-isomorphism $[Z'/\mu_{\mathbf r_x}]\,\simeq\, \mathfrak X'$, where
$\mathfrak X'\,=\,X'\times_X \mathfrak X$. So we get a local $X'$-morphism $Y'\longrightarrow \mathfrak X'$ that is a $G$-torsor.

We conclude the proof by showing that $\BHom^G_X(Y,\mathfrak X)\longrightarrow X$ is fully faithful. Since the same argument applies locally, we will prove that there is exactly one (iso)morphism between two global objects $q,q'\in \Hom^G_X(Y,\mathfrak X)$.
Because $q$ (respectively $q'$) is a $G$-torsor, it is $X$-isomorphic to $Y\longrightarrow [Y/G]$. By composing these isomorphisms, one obtains a $2$-commutative diagram~:

\newcommand{\eq}[1][r]
 {\ar@<-3pt>@{-}[#1]
  \ar@<-1pt>@{}[#1]|<{}="gauche"
  \ar@<+0pt>@{}[#1]|-{}="milieu"
  \ar@<+1pt>@{}[#1]|>{}="droite"
  \ar@/^2pt/@{-}"gauche";"milieu"
    \ar@/_2pt/@{-}"milieu";"droite"}
\[
	\xymatrix{
	& \mathfrak X\ar[rd] \eq[dd]&\\
	Y\ar[ru]^{q}\ar[rd]_{q'} && X \\
					  & \mathfrak X\ar[ru]&\\
	}
\]
By Lemma \ref{lem:rigidity}, one gets a $X$-isomorphism $q\simeq q'$. So $\BHom^G_X(Y,\mathfrak X)\longrightarrow X$ is full. 

To show it is faithful, let us fix $q: Y\longrightarrow \mathfrak X $ an object of $\BHom^G_X(Y,\mathfrak X)$.
We will check that $\Aut_X(q)$ is trivial. 
Since $q: Y\longrightarrow \mathfrak X $ is isomorphic to $q_0: Y \longrightarrow [Y/G]$, it is enough to show that $\Aut_X(q_0)$ is trivial. An automorphism of $q_0$ corresponds to an automorphism of the trivial right $G$-torsor $\pr_2:Y\times G\longrightarrow Y$ commuting with the action $a:Y\times G\longrightarrow Y$. Such an automorphism is given by $g_0\in G(Y)$ acting on the left of the second factor of $Y\times G$, and it commutes with $a$ if and only if $g_0$ acts trivially on $Y$. But $G$ acts freely, hence faithfully, on $Y\backslash\pi^{-1}(D)$, which is the complement of a divisor, hence $g_0=1$.
\end{proof}

\subsection{From torsors on the stack of roots to tamely ramified torsors}
\label{sec:stacky-to-tame}
The right context for most results in this section is the world of \emph{tame stacks}, in the sense of \cite{aov:tame_stacks}. Stack of roots are tame. The letter $\mathfrak X$ will mainly denote a tame stack, that we will always suppose locally noetherian.

The key result is the following variant of a result of A. Alper (\cite[Theorem 10.3]{alper:good_moduli}).

\begin{proposition}
\label{prop:alper_crit_torsors_local}
Let $G/S$ be an abelian affine group scheme of finite type, $\mathfrak X$ be a tame stack with good moduli space $X$ and consider two $G$-torsors $Y_i\longrightarrow \mathfrak X $ for $i\in\{ 1,2\}$. Assume that for a closed point $x\in \left|\mathfrak X\right|_0$, the restrictions to the residual gerbe ${Y_1}_{|\mathcal G_x}\longrightarrow \mathcal G_x$ and ${Y_2}_{|\mathcal G_x}\longrightarrow \mathcal G_x$ are isomorphic. Then there exists a fppf neighbourhood $X'\longrightarrow X$ of $\pi(x)$ in $X$ such that ${Y_1}$ and ${Y_2}$ are isomorphic over $X'\times_X \mathfrak X$.
\end{proposition}

\begin{remark}
\label{rem:rydh}
As David Rydh informed us, the Proposition is wrong if we don't assume $G$ to be abelian, even if $G$ is finite. We reproduce his counter-examples in Appendix \ref{sec:David_Rydh_counter_example}.
\end{remark}

Before proving this proposition, we now conclude the proof of Proposition \ref{prop:equiv}. So for a short while, the letter $\mathfrak X$ will again denote a stack of roots. Let $ Y \longrightarrow \mathfrak X$ be a $G$-torsor, where $\mathfrak X$ is the stack of roots, $G$ is an abelian finite group scheme, and $Y$ is a scheme. We have to show that \emph{fppf} locally on $X$, the morphism $Y\longrightarrow X$ to the moduli space is induced by a Kummer cover with ramification data $(\mathbf D,\mathbf r)$ along a group monomorphism defined on some extension $k'/k$.

Let $x\,\in\, \mathfrak X$ be a closed point. Since we can make an arbitrary base 
change to show the result, we can assume that $k(x)=k$. In a Zariski neighbourhood 
of $x$, we can write $\mathfrak X\,\simeq\, [Z/H]$, where $H\,=\,\mu_{\mathbf r_x}$, and 
$Z\longrightarrow X$ is a $H$-Kummer cover, with ramification data $(\mathbf 
D,\mathbf r)$ (see Remark \ref{rem:kum-quot}). This implies that $\mathcal G_x\simeq 
\B_{k(x)} H= \B H$, and that the monomorphism $i:\B H \longrightarrow \mathfrak X$ 
is a section of $Z: \mathfrak X \longrightarrow \B H $. Now, after enlarging $k$ 
again, we can assume that the composition $\B H\,\longrightarrow\, \mathfrak X 
\,\xrightarrow{Y}\, \B G$ comes from a group morphism $H\longrightarrow G$, that 
must be a monomorphism, since we assume that $Y$ is a scheme (hence $Y$ as a 
morphism $\mathfrak X \longrightarrow \B G$ is representable, and so is $\B H 
\longrightarrow \B G$).

We can finish the proof by considering the two $G$-torsors on $\mathfrak X$ given by $Y_1=Y$ and $Y_2= Y\circ i \circ Z$ (notice that $Y_2$ is just the torsor induced by $Z\longrightarrow \mathfrak X$ along the monomorphism $H\longrightarrow G$). Indeed, we have 
${Y_1}_{|\mathcal G_x}\simeq {Y_2}_{|\mathcal G_x}$
by construction, and Proposition \ref{prop:alper_crit_torsors_local} ensures that $Y_1$ and $Y_2$ are isomorphic over a \emph{fppf} neighbourhood of $\pi(x)$ in $X$. 

\begin{proof}[Proof of Proposition \ref{prop:alper_crit_torsors_local}]
Since $G$ is abelian, the sheaf $\BIsom_G(Y_1,Y_2)$ on $\mathfrak X$ is a 
$G$-torsor and the following proposition allows to conclude.
\end{proof}

\begin{proposition}
\label{prop:alper_crit_torsors_single}
Let $G/S$ be a group scheme of finite type, $\mathfrak X$ be a tame stack with moduli space $X$ and let $Y\longrightarrow \mathfrak X$ be a $G$-torsor. Assume that for a closed point $x\in \left|\mathfrak X\right|_0$, the restriction $Y_{|\mathcal G_x}\longrightarrow \mathcal G_x$ is trivial. Then there exists a fppf neighbourhood $X'\longrightarrow X$ of $\pi(x)$ in $X$ such that $Y\longrightarrow \mathfrak X$ is trivial over $X'\times_X \mathfrak X$.
\end{proposition}

We again postpone the proof of this last proposition, and first state an auxiliary result.

\begin{lemma}
\label{lem:restriction_trivial_open}
Let $\mathfrak X$ be a tame stack, $\mathcal E$ a locally free sheaf of finite rank on $\mathfrak X$. Assume that for a closed point $x\in \left|\mathfrak X\right|_0$, the restriction $\mathcal E_{|\mathcal G_x}$ is trivial. Then there exists an open substack $\mathfrak U\subset \mathfrak X$ such that $x\,\in\,\left|\mathfrak U\right|_0$ and for all $x'\in \left|\mathfrak U\right|_0$, the restriction $\mathcal E_{|\mathcal G_{x'}}$ is trivial.
\end{lemma}

\begin{proof}
First notice that for any $x'\,\in\,\left|\mathfrak X\right|_0$, $\mathcal E_{|\mathcal G_{x'}}$ is trivial if and only if the natural morphism $\pi^*\pi_*\mathcal E\longrightarrow \mathcal E$ is an epimorphism at $x'$. Namely, the only if direction is showed in \cite[Proof of Theorem 10.3]{alper:good_moduli}, and the if direction follows from the fact that $(\pi^*\pi_*\mathcal E)_{|\mathcal G_{x'}}$ is trivial, and for any gerbe the subcategory of finite direct sums of the trivial representation is stable by quotient.

Now, since $\mathcal E$ is of finite type, the locus where $\pi^*\pi_*\mathcal E\longrightarrow \mathcal E$ is an epimorphism is open, which concludes the proof.
\end{proof}

We finally prove Proposition \ref{prop:alper_crit_torsors_single} by using the standard Tannakian interpretation of torsors as tensor functors.

\begin{proof}[Proof of Proposition \ref{prop:alper_crit_torsors_single}]
	Denote by $F_Y:\Rep_k G\longrightarrow \Vect \mathfrak X$ the tensor functor associated to the $G$-torsor $Y\longrightarrow \mathfrak X$. If $V\in \Rep_k G$, we also write $\mathcal E_V$ for $F_Y(V)$. By hypothesis, the composite tensor functor $F_{Y_{|\mathcal G_x}}:\Rep_k G\longrightarrow \Vect \mathcal G_x$ is trivial.

	Since $G/S$ is of finite type, the Tannakian category $\Rep_k G$ admits a 
tensor generator $W$.
By Lemma \ref{lem:restriction_trivial_open}, there exists an open substack $\mathfrak U\longrightarrow \mathfrak X$ containing $x$ such that for all $x'\in \left|\mathfrak U\right|_0$, the vector bundle
$(\mathcal E_{W})_{|\mathcal G_{x'}}$ is trivial. Since each $F_{Y_{|\mathcal G_{x'}}}$ is tensor and exact, it follows that for all $x'\in \left|\mathfrak U\right|_0$, and all representations $V \in \Rep_k G$, the vector bundle $(\mathcal E_V)_{|\mathcal G_{x'}}$ is trivial.

As Zariski spaces we have $\left|\mathfrak X\right|\,=\,\left|X\right|$, so
all what remains to be proved is the following. Let $Y\longrightarrow \mathfrak X$ be a
$G$-torsor such that for all $x\,\in\, \left|\mathfrak X\right|_0$, and all representations $V
\in \Rep_k G$, the vector bundle $(\mathcal E_V)_{|\mathcal G_x}$ is trivial.
Then there exists a \emph{fppf} cover $X'\longrightarrow X$ such that $Y_{X'}
\longrightarrow \mathfrak X_{X'}$ is trivial.
Let $\Vect(\mathfrak X)_{triv}$ the full sub-category of
$\Vect(\mathfrak X)$ consisting of objects $\mathcal E$ such that for
all $x\in \left|\mathfrak X\right|_0$, the vector bundle
$\mathcal E_{|\mathcal G_x}$ is trivial. This is clearly a tensor
sub-category. According to \cite[Theorem 10.3]{alper:good_moduli}, the 
functor $\pi^*$ induces an equivalence $\Vect(X)\longrightarrow
\Vect(\mathfrak X)_{triv}$, with inverse equivalence $\pi_*$. This last functor
is exact since $\mathfrak X$ is tame. Since $\pi^*$ is a (strong) tensor
functor, $\pi_*$ is endowed with a (weak) tensor structure, that is strong because
$\pi^*$ is an equivalence. By hypothesis, the functor $F_Y$ factors through
$\Vect(\mathfrak X)_{triv}$, hence the functor $\pi_*\circ F_Y : \Rep_k G
\longrightarrow \Vect X$ is well defined, tensor, and exact. But this precisely
means that the $G$-torsor $Y\longrightarrow \mathfrak X$ descends along
$\pi:\mathfrak X\longrightarrow X$.
\end{proof}

\subsection{A criterion to be ab-uniformizable}
\label{sub:a_criterion_to_be_ab_uniformizable}

It follows from Proposition \ref{prop:equiv} that the first assertion in the equivalence of Theorem \ref{crit_exist} exactly means that $\mathfrak X$ is ab-uniformizable, in the following sense.

\begin{definition}[]
\label{def:ab-unif}
An algebraic stack $\mathfrak X$ is ab-uniformizable if there exists a $G$-torsor $Y\longrightarrow \mathfrak X$, where $Y$ is a $k$-algebraic space, and $G$ is an abelian finite $k$-group scheme.
\end{definition}

\begin{remark}
\label{rem:ab-unif}
If $\mathfrak X$ is a stack of roots over a $k$-scheme, Remark \ref{rem:equiv} \eqref{rem:equiv-alg-space} shows that $\mathfrak X$ is ab-uniformizable if and only if there exists a $G$-torsor $Y\longrightarrow \mathfrak X$, where $Y$ is a $k$-scheme, and $G$ is an abelian finite $k$-group scheme.
\end{remark}

We will need the following ``abelian'' variant of the main result of \cite{bisbor:fundamental_gerbe}. Recall from \cite{bv:nori_gerbe} that an algebraic stack $\mathfrak X/k$ is called inflexible if it admits a fundamental gerbe.

\begin{proposition}
\label{prop:char_ab_unif}
Let $\mathfrak X/k$ be an inflexible proper stack of finite type and with finite inertia. Then $\mathfrak X$ is ab-uniformizable if and only if for any closed point $x$, any representation $V$ of $\mathcal G_x$ is a subquotient of the restriction of a basic essentially finite vector bundle on $\mathfrak X$ along $\mathcal G_x\,\longrightarrow\, \mathfrak X$.
\end{proposition}

We first state and prove an auxiliary lemma, that was suggested to us by the referee, and generalizes the proof of \cite[Proposition 9]{bisbor:fundamental_gerbe}. The reader can find the precise definition of a projective system of affine gerbes in \cite[Definition 3.3] {bv:nori_gerbe}.

\begin{lemma}
\label{lem:sugg-referee}
Let $\mathfrak X$ be a quasi-compact algebraic stack over $k$ with finite inertia, $I$ a cofiltered $2$-category, $\Pi=\varprojlim_{i\in I} \Pi_i$ a projective limit of affine gerbes over $k$, and $\mathfrak X\longrightarrow \Pi$ a morphism. Then $\mathfrak X\longrightarrow \Pi$ is faithful if and only if there exists an object $i$ of $I$ such that $\mathfrak X\longrightarrow \Pi_i$ is faithful.
\end{lemma}

\begin{proof} The `if' condition is clear, let us show the `only if' one.

	Let $I_{\mathfrak X/\Pi_i}=\ker\left(I_{\mathfrak X}\longrightarrow \left(I_{\Pi_i}\right)_{|\mathfrak X} \right)$ be the relative inertia stack. The morphism $\mathfrak X\longrightarrow \Pi_i$ is faithful if and only if $I_{\mathfrak X/\Pi_i}$ is trivial as a group scheme over $\mathfrak X$ (\cite[Proposition 15]{bisbor:fundamental_gerbe}). Since $I_{\mathfrak X/\Pi_i}$ is a finite group scheme over $\mathfrak X$, we can write $I_{\mathfrak X/\Pi_i}=\Spec \mathcal A_i$, where $\mathcal A_i$ is a quasi-coherent sheaf of algebras, that is finitely generated as a sheaf of $\mathcal O_{\mathfrak X}$-modules. The structural morphism $\mathcal O_{\mathfrak X}\longrightarrow \mathcal A_i$ is injective (since it admits a retraction corresponding to the unit section of $I_{\mathfrak X/\Pi_i}$). Hence the locus $\mathfrak U_i$ where $I_{\mathfrak X/\Pi_i}$ is trivial is the zero-locus of $\mathcal A_i/\mathcal O_{\mathfrak X}$, which is open. 
	
	Because we assume that $\mathfrak X\longrightarrow\Pi$ is faithful, we have that $I_{\mathfrak X/\Pi}$ is trivial, which implies that $\varinjlim_{i\in I} \mathcal A_i/\mathcal O_{\mathfrak X}= 0 $. Since for any arrow $j\longrightarrow i$ in $I$, we have that $I_{\mathfrak X/\Pi_j}$ is a closed subscheme of $I_{\mathfrak X/\Pi_i}$, the corresponding morphism $\mathcal A_i/\mathcal O_{\mathfrak X}\longrightarrow \mathcal A_j/\mathcal O_{\mathfrak X}$ is an epimorphism. Now if $x$ is any point of $\mathfrak X$, as $I$ is cofiltered, and each $\mathcal A_i$ is of finite type, the fact $\varinjlim_{i\in I} \mathcal A_i/\mathcal O_{\mathfrak X}$ is trivial at $x$ implies that there exists $i$ in $I$ such that $\mathcal A_i/\mathcal O_{\mathfrak X}$ is trivial at $x$. Hence $\cup_{i\in I} \mathfrak U_i=\mathfrak X$. Since $\mathfrak X$ is quasi-compact and $I$ is cofiltered, there exists $i$ in $I$ such that $\mathfrak U_i=\mathfrak X$.
\end{proof}

\begin{proof}[Proof of Proposition \ref{prop:char_ab_unif}]
From Lemma \ref{lem:sugg-referee} and \cite[Proposition 8]{bisbor:fundamental_gerbe} it follows that $\mathfrak X$ is ab-uniformizable if and only if $\mathfrak X\longrightarrow\pi_{\mathfrak X/k}^{ab}$ is faithful. Applying \cite[Lemma 12]{bisbor:fundamental_gerbe}, this is equivalent to : for any closed point $x$ of $\mathfrak X$, the morphism $\mathcal G_x\longrightarrow \pi_{\mathfrak X/k}^{ab}$ is faithful. By \cite[Proposition 21, Lemma 14]{bisbor:fundamental_gerbe}, this means that any representation $V$ of $\mathcal G_x$ is a subquotient of the restriction of a basic essentially finite vector bundle on $\mathfrak X$ along $\mathcal G_x\,\longrightarrow\, \mathfrak X$. 
\end{proof}

\begin{remark}
\label{rem:erratum}
We must warn the reader that \cite[Proposition 21]{bisbor:fundamental_gerbe} is incorrectly stated and holds for \emph{algebraic} gerbes only. This confusion appears several times, for instance \cite[Proposition 15]{bisbor:fundamental_gerbe} is valid for \emph{algebraic} stacks. But everything is correct if whenever a non-algebraic stack appears (for instance a fundamental gerbe) one replaces `representable' by `faithful'.
\end{remark}

\section{Proof of Theorem \ref{crit_exist}}

As already noticed, the first assertion in the equivalence of Theorem \ref{crit_exist} exactly means that $\mathfrak X$ is ab-uniformi\-zable. We will now proceed to show that the second assertion is equivalent.

In view of the assumptions we made on the scheme $X$, the stack $\mathfrak X$ is geometrically connected (\cite[Corollaire 3.7]{bor:corr}) and geometrically reduced (\cite[\S 1, Proposition 1.7.2]{gm:tame}), hence inflexible (\cite[Proposition 5.5]{bv:nori_gerbe}). It is also proper, of finite type (\cite[Corollaire 3.6]{bor:corr}), and has finite inertia (see \cite[\S 4.2.3]{bor:corr} for the description of the inertia stack of a stack of roots). Hence we can use Proposition \ref{prop:char_ab_unif}.
that asserts that $\mathfrak X$ is ab-uniformizable if and only if for any closed point $x$, any representation $V$ of $\mathcal G_x$ is a subquotient of the restriction of a basic essentially finite vector bundle on $\mathfrak X$ along $\mathcal G_x\longrightarrow \mathfrak X$. Of course, this condition is empty outside of the support of $D$, where the residual gerbes are trivial.

Before studying $\mathcal G_x$, let us recall how vector bundles on $\mathfrak X$ are interpreted. 

\begin{remark}
\label{rem:correspondence}
To each vector bundle $\mathcal F$ on $\mathfrak X$, one can associate a parabolic 
vector bundle $\mathcal E_\cdot$ on $(X,\mathbf D)$ with weights in 
$\frac{1}{\mathbf r}\mathbb Z^I$ in the following way : if $\mathbf l$ belongs to 
$\mathbb Z^I$, we define $\mathcal E_{\frac{\mathbf l}{\mathbf r}} =\pi_* 
\left(\mathcal F\otimes_{\mathcal O_{\mathfrak X}}{\mathcal N^{\otimes \mathbf 
l}}^{\vee}\right)$. This functor gives a tensor equivalence between (essentially 
finite) vector bundles on $\mathfrak X$ and (essentially finite) parabolic vector 
bundles on $X$ with weights in $\frac{1}{\mathbf r}\mathbb Z^I$ (\cite[Th\'eor\`eme 
2.4.7]{bor:rep})\footnote{It should be clarified that we rely here crucially on the fact
that $D$ is a simple normal crossings divisor. By using the formalism of 
\cite{bv:par_sheaves}, one can extend this correspondence to the case of a general 
normal crossings divisor.}
\end{remark}

To proceed further, we must describe the residual gerbe $\mathcal G_x$ at a closed point $x$. According to Lemma \ref{lem:residual-gerbe}, this gerbe is canonically banded by $\mu_{\mathbf r_x}$ (since the band is abelian, this just means that there is a canonical isomorphism $I_{\mathcal G_x}\simeq \left(\mu_{\mathbf r_x}\right)_{\mathcal G_x}$). It is moreover neutral, but it is not canonically neutralized (in other words, there is no canonical choice of a section of the residual gerbe). 
So we will avoid choosing a specific section of $\mathcal G_x$, and just use the fact that the 
category $\Vect \mathcal G_x$ is semi-simple and that any object is a direct sum of 
$\mathcal N_{|\mathcal G_x}^{\otimes \mathbf l}$, for $\mathbf 0 \leq \mathbf l 
<\mathbf r_x$. From this we deduce that $\mathfrak X$ is ab-uniformizable if and 
only if for any closed point $x$, and for any $\mathbf 0 \leq \mathbf l <\mathbf r_x$, 
there exists an essentially finite basic vector bundle $\mathcal F$ on $\mathfrak X$ such that
$\Hom_{\mathcal G_x}(\mathcal N_{|\mathcal G_x}^{\otimes \mathbf l},\mathcal F_{|\mathcal G_x})\neq \{0\}$.
%$\Hom_{\mathcal G_x}(\mathcal F_{|\mathcal G_x},\mathcal N_{|\mathcal G_x}^{\otimes 
%\mathbf l})\neq \{0\}$.

So the only thing that remains to be done is to show that, if $\mathcal E_\cdot$ is the parabolic sheaf associated to $\mathcal F$ via the correspondence, and $x$ is a closed point in the support of $D$, then for any $\mathbf 0 \leq \mathbf l <\mathbf r_x$, we have $\Hom_{\mathcal G_x}(\mathcal N_{|\mathcal G_x}^{\otimes \mathbf l},\mathcal F_{|\mathcal G_x})\neq \{0\}$ if and only if $\mathcal E_\cdot$ admits $\mathbf l /\mathbf r $ as a weight at $x$. %Consider ${\mathfrak D}\longrightarrow D$ the gerbe of $\mathbf r$-th roots of $\mathcal O(\mathbf D)_{|D}$, or equivalently, the reduced substack associated to $\mathfrak X _{|D}$. It fits into the following commutative diagram

For $J\subset I$, set $D_J=\cap_{i\in J} D_i$ and $D_J^\circ=D_J\backslash\cup_{K\supsetneq J} D_K$, and similarly by substituting $D$ with $\mathfrak D$. The morphism  $\mathfrak D_J^\circ \to D_J^\circ$ is a gerbe banded by $\mu_{\mathbf r_J}=\prod_{i\in J} \mu_{r_i}$ and since $x\in D_{I_x}^\circ$ this shows that in the following commutative diagram:

\[
	\xymatrix{
		\mathcal G_x\ar[r]\ar[d] & \mathfrak D_{I_x} \ar[r]^{j_{I_x}}\ar[d]_{p_{I_x}} &\mathfrak X\ar[d]^{\pi} \\
\spec k(x) \ar[r] &  D_{I_x} \ar[r]_{i_{I_x}} & X
	}
\]
%where only 
the left hand square is Cartesian. By base change, 
%which holds because ${\mathcal D}\longrightarrow D$ is a gerbe banded by a linearly reductive group, 
we get : $$\Hom_{\mathcal G_x}(\mathcal N_{|\mathcal G_x}^{\otimes \mathbf l},\mathcal F_{|\mathcal G_x})\simeq \left({p_{I_x}}_*(\mathcal F\otimes_{\mathcal O_{\mathfrak X}}{\mathcal N^{\otimes \mathbf l}}^\vee)_{|\mathfrak D_{I_x}}\right)\otimes k(x)\; .$$ But from the exact sequence

\[  \left(\bigoplus_{i\in I_x}  \mathcal N_i\right)^\vee \longrightarrow \mathcal O_{\mathfrak X} \longrightarrow (j_{I_x})_*\mathcal O_{\mathfrak D_{I_x}}\longrightarrow 0 \]
we deduce by exactness of $\pi_*$ the exact sequence

\[
 \bigoplus_{i\in I_x}\mathcal E_{\frac{\mathbf l+\mathbf e_i}{\mathbf r} }   \longrightarrow \mathcal E_{\frac{\mathbf l}{\mathbf r}} \longrightarrow {i_{I_x}}_* {p_{I_x}}_*(\mathcal F\otimes_{\mathcal O_{\mathfrak X}}{\mathcal N^{\otimes \mathbf l}}^\vee)_{|\mathfrak D_{I_x}}\longrightarrow 0 
\]
hence our claim, and the theorem follows.

\appendix

\section{Basic objects in Tannakian categories}
\label{sec:abelian_objects_in_tannakian_categories}

\subsection{Abelianization of affine group schemes}
\label{sub:abelianization_of_affine_group_schemes}

We follow partly the exposition given in \cite[Chapter 10]{wat:aff}.
Let $k$ be a field and $G/k$ an affine group scheme. Denote by $H$ its Hopf algebra. The commutator morphism $G^{2n}\longrightarrow G$ sending $(x_1,y_1,\cdots,x_n,y_n)$ to
$[x_1,y_1]\cdots [x_n,y_n]$ corres\-ponds to a ring morphism $H\longrightarrow H^{\otimes 2n}$ whose kernel will be denoted by $I_n$.

\begin{definition}[]
\label{def:derived_subgroup}
The derived subgroup $D(G)$ of $G$ is the closed normal subgroup defined by the ideal $\cap_{n\in \mathbb N^*} I_n$. 
\end{definition}

To check that $D(G)$ is a normal subgroup is straightforward. Thus, by Chevalley's theorem, the 
sheaf $G/D(G)$ is representable by an affine group scheme. Moreover, it is also easy to verify that $G/D(G)$ is abelian and that the morphism $G\longrightarrow G/D(G)$ is universal among morphisms from $G$ to abelian affine group schemes over $k$. 

We will need the fact that the formation of $D(G)$ commutes with base change, in the following sense.

\begin{lemma}
\label{lem:derived_subgroup_base_change}
If $A/k$ is a $k$-algebra, then as ideals of $H\otimes_k A$ : 
$$ \left(\cap_{n\in \mathbb N^*} I_n\right)\otimes_k A= \cap_{n\in \mathbb N^*}\left( I_n \otimes_k A \right)\, .$$
\end{lemma}

\begin{proof}
The inclusion from left to right is obvious, and it is enough to consider a basis of $A/k$ to show the opposite inclusion.
\end{proof}

\subsection{Abelian gerbes}
\label{sub:abelian_gerbes}

Let $k$ be a field, we denote as usual $S=\spec k$. Following \cite[\S 3]{bv:nori_gerbe}, we will use \emph{affine gerbe} for an affine fpqc gerbe with an affine chart, and an affine diagonal. These gerbes are called Tannakian in \cite{rivano:cat_tan}.

\begin{definition}
\label{def:abelian_gerbe}
The affine gerbe $\mathcal G/k$ will be called \emph{abelian} if there is an extension $k'/k$, and an object $\xi$ of $\mathcal G(k')$ such that $\Aut_{k'}(\xi)$ is abelian.
\end{definition}

Since a form of an abelian group is abelian, if this holds for one object, this is true for all objects. Let us reformulate this canonically.

Let $I_\mathcal{G}= \mathcal{G} \times_{\mathcal{G}\times_S \mathcal{G}}\mathcal{G}$ be the inertia stack of $\mathcal G$. 
The gerbe $\mathcal{G}$ is abelian if and only if $I_\mathcal{G}$, as a group scheme over $\mathcal G$, is abelian.

\begin{proposition}\label{prop:car_ab_gerbe}
Let $\phi:\mathcal{G}\longrightarrow S$ be the structural morphism.
The gerbe $\mathcal G/k$ is abelian if and only if its inertia stack descends along $\phi$ as a group scheme to $S= \spec{k}$.
\end{proposition}
The last assertion means that there exists a $k$-affine group scheme $G$ and an isomorphism \( \mathcal{G} \times_{S} G \simeq I_{\mathcal G} \) of group schemes over \( \mathcal{G} \). We don't assume a priori that $G$ is abelian.

\begin{proof}
	For the ``only if '' part: it is straightforward to check, more generally, that the center $\Z(I_{\mathcal G})$ of the inertia stack always descends along $\phi:\mathcal{G}\longrightarrow S$. Indeed, if $\xi$ and $\nu$ are two objects of $\mathcal G(S)$, $h:\xi \longrightarrow\xi$ is in $\Z(\Aut(\xi))$, and $v:\nu\longrightarrow \xi$ is an isomorphism, then $v^{-1}hv:\nu \longrightarrow \nu$ is in $\Z(\Aut(\nu))$ and is independent of $v$. Since such a $v$ exists locally, we get a canonical isomorphism $\Z(\BAut(\xi))\,\simeq\,\Z(\BAut(\nu))$. For the ``if'' part: assume that $I_\mathcal G$ descends along $\phi:\mathcal{G}\longrightarrow S$ into a $k$-group scheme $G$. Thus for any extension $k'/k$ and any object $x'\in \mathcal G(k')$, there is an automorphism $G_{k'}\simeq \Aut_{k'}(x')$, natural in $x'$. This implies that the conjugacy action of $\Aut_{k'}(x')$ on itself is trivial, hence the result.
\end{proof}

\subsection{Abelianization of affine gerbes}
\label{sub:abelianization_of_affine_gerbes}

Let $ \mathcal{G}/k$ be an affine gerbe. Since the inertia stack $I_{ \mathcal{G}}$ is defined over $\mathcal G$, and not over a field, some descent argument is needed to show that the derived subgroup $D(I_{ \mathcal{G}})$ makes sense.

We proceed as follows : let $k'/k$ be an extension, and $\xi$ an object of $\mathcal G(k')$. Denote $U=\spec k'$ and $G'\,=\,\Aut_{k'}(\xi)$. We want to show that $D(G')$ descends along $U\xrightarrow{\xi} \mathcal{G}$. Let $R=U\times_{ \mathcal{G}}U$, this amounts to show the following lemma.

\begin{lemma}
\label{lem:derived_inertia_stack}
With the notations as above, the canonical isomorphism $\phi:\pr_1^*G' \simeq \pr_2^*G'$ sends $\pr_1^* D(G')$ to $\pr_2^* D(G')$.
\end{lemma}

\begin{proof}
Let $\gamma$ the canonical isomorphism $\gamma:\pr_1^*x \simeq \pr_2^*x$ on $R$. Then $\phi$ identifies with the morphism of conjugation by $\gamma$ :
\[ c_\gamma : \Aut_R (\pr_1^*x)\longrightarrow \Aut_R (\pr_2^*x)\;\; . \] 
In particular, $\phi$ is a group morphism, hence $\phi$ commutes with the two pullbacks of the commutator map $\pr_i^*\left(G'^{2n}\longrightarrow G'\right )$ for $i=1,2$, in the obvious sense.

Let $H'$ be the Hopf algebra of $G'$, let $A=\mathcal O_R (R)$, endowed with the two ring morphisms $i_1,i_2 : k'\longrightarrow A$ corresponding respectively to the two projections $\pr_1,\pr_2:R\longrightarrow U$. Then $\phi$ translates into a ring morphism $\phi_\# : H'\otimes_{k', i_2} A \longrightarrow H'\otimes_{k', i_1} A $. Since, as we have just seen, $\phi$ is compatible with the commutators maps, the morphism $\phi_\#$ sends $I_n\otimes_{k', i_2} A$ to $I_n\otimes_{k', i_1} A$. According to Lemma \ref{lem:derived_subgroup_base_change}, $\phi_\#$ sends then $\left(\cap_{n\in \mathbb N^*}I_n\right) \otimes_{k', i_2} A$ to $\left(\cap_{n\in \mathbb N^*}I_n\right)\otimes_{k', i_1} A$ which concludes the proof.
\end{proof}

Thus the subgroup $D(I_{ \mathcal{G}})$ makes sense, and the same Lemma \ref{lem:derived_subgroup_base_change} shows that it is independent of the choice of the element $\xi$ in $\mathcal{G}(k')$. The abelianization of $ \mathcal{G}$ will be the rigidification $ \mathcal{G}\thickslash D(I_{ \mathcal{G}})$. 

More generally, if $\mathfrak X$ is any stack over $S$, and $G< I_{\mathfrak X}$ is a subgroup of its inertia group, we say that a stack morphism $\mathfrak X\longrightarrow \mathfrak Y$ is trivial on $G$ if the composite morphism of $\mathfrak X$-group schemes $G\longrightarrow I_{\mathfrak X}\longrightarrow\mathfrak X \times_{\mathfrak Y} I_{\mathfrak Y}$ is trivial. Naturally, a rigidification along $G$ is a morphism $\mathfrak X\longrightarrow \mathfrak X \thickslash G$ trivial on $G$ and universal for this property.

Usually, rigidification of \emph{fppf} stacks is conducted in two steps : one first mod out the morphisms by $G$, and then stackify the resulting prestack (see for instance \cite[Appendix A]{aov:tame_stacks} for algebraic stacks). In our situation, more care is needed : we are forced to consider the \emph{fpqc} topology for the needs of Tannaka theory, or simply because our gerbes may not be of finite type, but on the other hand stackification (even sheafification) is well-known to involve set-theoretic issues with this topology. But there is a workaround : one can work explicitly with groupoids, since these admit an explicit stackification process via torsors.

\begin{proposition}
\label{prop:abelianization_gerbe}
Let $ \mathcal{G}/S$ be an affine gerbe. Then the rigidification $ \mathcal{G}\longrightarrow \mathcal{G}\thickslash D(I_{ \mathcal{G}})$ exists, and is an universal morphism to an affine abelian gerbe.
\end{proposition}

\begin{remark}
\label{rem:abelianization-gerbe}\mbox{}
\begin{enumerate}
	\item 
		The existence of the abelianization of an affine gerbe follows from the general machinery developed in \cite{bv:fundamental_gerbes}, but since this is not really required, we prefer to work it out explicitly. An even more direct approach is given in \cite{tz:alg-nori-fund-gerbes} Definition B.9 (for the largest pro-\'etale and pro-local quotient of an affine gerbe, by this works with the largest abelian quotient as well). However, it relies heavily on Tannaka duality. 

\item To keep notations compact, we will write $ \mathcal{G}^{ab}$ for $\mathcal{G}\thickslash D(I_{ \mathcal{G}})$.
\end{enumerate}

\end{remark}

\begin{proof}[{Proof of Proposition \ref{prop:abelianization_gerbe}}]
	
We switch to the point of view of (non empty, transitive) groupoids : let $k'/k$ be an extension, and $\xi$ an object of $\mathcal G(k')$. As before, we denote $U=\spec k'$, $R=U\times_{ \mathcal{G}}U$. We can recover $\mathcal G$ as the stack of $R\rightrightarrows U$-torsors (see \cite[Proposition 2.12]{breen:tannaka}). Let $G'=\Aut_{k'}(\xi)$, then $R$ is naturally a $(\pr_1^*G',\pr_2^*G')$-bitorsor over $U\times_{k}U$. Set : 
$$\widetilde R \,=\, \pr_1^*D(G')\backslash R \slash \pr_2^*D(G')$$

From this presentation, it is easy to check that the composition law of $R$ descends to $\widetilde R$, in other words we get in new groupoid $\widetilde{R}\rightrightarrows U$. It is clear that it is again non empty and transitive, so that the stack of $\widetilde{R}\rightrightarrows U$-torsors is a gerbe, that we will of course denote by $\mathcal{G}\thickslash D(I_{ \mathcal{G}})$. The fact that $\mathcal{G}\thickslash D(I_{ \mathcal{G}})$ is abelian and satisfies the required universal property is clear by construction. The only thing that remains to be checked is that $\mathcal{G}\thickslash D(I_{ \mathcal{G}})$ is an affine gerbe. From the fact that left and right actions on $R$ commute, we get that $\widetilde R = \pr_1^*D(G')\backslash R$. So ${\widetilde R}\longrightarrow U\times_{k}U$ is a $\pr_1^* G'/\pr_1^*D(G')$-torsor. Since $G'/D(G')$ is affine by Chevalley's theorem, and $U\times_{k}U$ is affine as well, wet get that $\widetilde R$ is affine, so $\mathcal{G}\thickslash D(I_{ \mathcal{G}})$ is an affine gerbe. 
\end{proof}

\subsection{Basic Tannakian categories}
\label{sub:basic_tannakian_categories}

It turns out that Proposition \ref{prop:car_ab_gerbe} enables to give a Tannakian characterization of Tannakian categories whose holonomy gerbe is abelian.

To find a substitute for the inertia stack, we have to follow Deligne's idea (see \cite[\S 5]{del:group_fond}) that you can do algebraic geometry within any Tannakian category. The only definition we will need, in fact, is the following.

\begin{definition}
\label{def:group_within_tannaka}
An affine group scheme $\pi$ within a Tannakian category $\mathcal C$ is (the spectrum of) an ind Hopf algebra in $\mathcal C$.
\end{definition}

It is clear that this definition is functorial in $\mathcal C$ in the obvious sense: if $F:\mathcal C\longrightarrow \mathcal{C'}$ is a symmetric tensor functor between Tannakian categories, and $\pi$ is an affine group scheme in $\mathcal{C}$, then $\mathcal{\pi'}=F(\mathcal{\pi})$ is an affine group scheme in $\mathcal{C'}$. The main example is the following.

\begin{definition}[\cite{del:group_fond},\S 6]
\label{def:fund_group_tannaka}

The fundamental group $\pi(\mathcal C)$ of a Tannakian category $\mathcal C$ is the affine group scheme in \( \mathcal{C} \) verifying for any fiber functor $\omega :\mathcal C\longrightarrow \Vect k'$ over an arbitrary extension $k'/k$:
\[ \omega (\pi(\mathcal C)) \simeq \Aut^{\otimes}(\omega)\] 
in a functorial way.
\end{definition}
In others words, $\omega (\pi(\mathcal C))$ is the fundamental group of $\mathcal C$ based at $\omega$.

\begin{proposition}
\label{prop:corr_group_schemes}
Let $\mathcal C$ be a Tannakian category, and $\mathcal G$ the gerbe of its fiber functors.
\begin{enumerate}
	\item There is a natural equivalence between affine group schemes within $\mathcal C$ and affine group schemes over $\mathcal G$.
	\item Under this equivalence, the fundamental group $\pi(\mathcal C)$ is sent to the inertia stack $I_{\mathcal G}$.
\end{enumerate}
\end{proposition}

\begin{proof}
\begin{enumerate}
	\item See \cite[\S 5.8]{del:group_fond} for the neutral case. In the non neutral situation, to choose an affine group scheme over $\mathcal G$ amounts to choose a quasi-coherent Hopf algebra $A$ over $\mathcal G$. But any quasi-coherent sheaf over $\mathcal G$ is locally free (because this is true in any chart $\Vect k'\longrightarrow \mathcal G$). So the data of $A$ boils down to the data of an Hopf algebra in the category of ind objects $\Ind \mathcal{C} $. 
	\item By definition 
		$$I_\mathcal{G}= \mathcal{G} \times_{\mathcal{G}\times_S \mathcal{G}}\mathcal{G}=\Aut_\mathcal{G}(\mathcal{G}\xrightarrow{\id}\mathcal{G})\;\;$$
		By Tannaka duality and the equivalence of the first point, this translates in $\Aut^\otimes_\mathcal{C}(\mathcal{C}\xrightarrow{\id}\mathcal{C})$, which clearly has the universal property required by Definition \ref{def:fund_group_tannaka}.
\end{enumerate}
\end{proof}

\begin{definition}
\label{def:struct_tannaka}
Let \( \mathcal{C} \) be a Tannakian category, and $\mathds 1$ its neutral object. The structure morphism of \( \mathcal{C} \) is the functor 
$$\cdot \otimes_k \mathds 1 : \Vect k\longrightarrow \mathcal C$$
given by $ V \longmapsto V\otimes_k \mathds 1$.
\end{definition}

One can now translate Proposition \ref{prop:car_ab_gerbe}:
\begin{definition}
\label{def:basic_tannaka}
A Tannakian category \( \mathcal{C} \) is called \emph{basic} if there exists an affine group scheme $G$ over $k$ such that $G\otimes_k \mathds 1 \simeq \pi(\mathcal C)$ as affine group schemes within $\mathcal C$.
\end{definition}

\begin{proposition}
	\label{prop:car_basic_tannakian}
Let $\mathcal C$ be a Tannakian category, and $\mathcal G$ the gerbe of its fiber functors. Then $\mathcal C$ is basic if and only if $\mathcal G$ is abelian.
\end{proposition}

\begin{proof}
	This follows from the definitions, Proposition \ref{prop:car_ab_gerbe} and Proposition \ref{prop:corr_group_schemes}.
\end{proof}

\subsection{Basic objects}
\label{sub:basic_objects}

Let \( \mathcal{C} \) be a Tannakian category, and $X \in \Ob \mathcal{C} $. We denote by \( \mathcal{C}_X=<X> \) the sub-Tannakian category generated by $X$ and by $\pi(X)$ its fundamental group in $\mathcal C$.

\begin{definition}
\label{def:basic_object}
The object $X$ of the Tannakian category \( \mathcal{C} \) is \emph{basic} if there exists an affine group scheme $G$ over $k$ such that $G\otimes_k \mathds 1 \simeq \pi(X)$ as affine group schemes within $\mathcal C$.
\end{definition}

We will write $\mathcal G_X$ for the holonomy gerbe of $X$ (that is, for the fundamental gerbe of \( \mathcal{C}_X \) ).

\begin{proposition}
\label{prop:car_basic_objects}
Let \( \mathcal{C} \) be a Tannakian category, and $X \in \Ob \mathcal{C} $. Then $X$ is basic if and only if $\mathcal G_X$ is abelian.
\end{proposition}

\begin{proof}
	This follows from Proposition \ref{prop:car_basic_tannakian} applied to the holonomy gerbe \( \mathcal{C}_X \).
\end{proof}

\begin{proposition}
\label{prop:basic_objects_tannaka_category}
Let \( \mathcal{C} \) be a Tannakian category. Then the full sub-category \( \mathcal{C}^{bas} \) of its basic objects is Tannakian and basic. Moreover any rigid tensor functor $\mathcal D \longrightarrow \mathcal{C} $ from a basic Tannakian category $\mathcal D$ factors uniquely through \( \mathcal{C}^{bas} \).
\end{proposition}

\begin{proof}
	This a translation of Proposition \ref{prop:abelianization_gerbe} via the Tannaka correspondence.
\end{proof}

\section{Counter-examples due to David Rydh}

\label{sec:David_Rydh_counter_example}

We reproduce two (families of) counter-examples kindly provided to us by David Rydh that shows that Proposition \ref{prop:alper_crit_torsors_local} fails if we don't assume $G$ to be abelian.

\subsection{First family}
\label{sub:first_family}

Let $X=\spec A$ be an affine scheme of positive characteristic $p$, and $s$ a non invertible regular element of $A$, defining a principal effective Cartier divisor $D$ on $X$. Set $\mathfrak X=\sqrt[p]{D/X}$, and let $\pi:\mathfrak X\to X$ be the canonical morphism. According to Lemma \ref{lem:kum-quot} we know that $\mathfrak X \simeq [Z/\mu_p]$, where $Z=\spec\left(\frac{ A[t]}{t^p-s}\right)$. In particular, there is a canonical morphism $\mathfrak X \xrightarrow{Z} \B\mu_p$.

Let $(\mathcal L,\tau)$ be the canonical $p$-th root of $(\mathcal O_X(D),s_D)$ on $\mathfrak X$, where $s_D$ is the section defined by $s\in A$. Let $L=\Spec\left(\Sym_{\mathfrak X} \mathcal L\right)$ be the corresponding line bundle. Since $D$ is defined by a global section of $\mathcal O_X$,  there is a canonical isomorphism $\mathcal L^{\otimes p}\simeq \pi^*\mathcal O_X(D) \simeq \mathcal O_{\mathfrak X}$, and $Z\simeq\Spec\left(\frac{\Sym_{\mathfrak X} \mathcal L}{\mathcal L^{\otimes p}\simeq \mathcal O_{\mathfrak X}}\right)$ (see \cite[Preuve du Th\'eor\`eme 3.4]{bor:corr}) . Thus we have that $L=Z\wedge^{\mu_p} \mathbb G_a$. Since $L^{-1}$ and not $L$ has a canonical section, it will be more convenient to consider the opposite torsor $\mathfrak X \xrightarrow{Z^{-1}} \B\mu_p$ instead of $Z$. So from now on, the only map $\mathfrak X \to \B\mu_p$ considered (in particular in fiber products) is defined by $Z^{-1}$.

Let $G=\mathbb \alpha_p \rtimes \mu_p$. The stack $\mathcal G=\mathfrak X\times_{\B\mu_p} \B G$ is a gerbe on $\mathfrak X$, neutralized by the morphism $\sigma$ associated to the canonical section $\mu_p \to G$. It is easy to check that $\BAut_{\mathfrak X}(\sigma)\simeq K$, where $K=Z^{-1}\wedge^{\mu_p}\alpha_p $ is the form of $\alpha_p$ that one gets by twisting by $Z^{-1}\rightarrow \mathfrak X$. Thus $\mathcal G \simeq \B_{\mathfrak X}K$ 	and  lifting the structure group of the $\mu_p$-torsor $Z^{-1}\rightarrow \mathfrak X$  along the projection $G\to \mu_p$ amounts to define a $K$-torsor on $\mathfrak X$. By this correspondence, the trivial torsor corresponds to the canonical lifting defined by the section $ \mu_p \to G$.

 The group $K$ fits into an exact sequence of $\mathfrak X$-group schemes : 
 \[ 0 \rightarrow K \rightarrow L^{-1} \xrightarrow{F} \left(L^{-1}\right)^{(p/\mathfrak X)} \rightarrow 0 \] 
where $F$ is the Frobenius morphism of $L^{-1}$ relative to $\mathfrak X$. The isomorphism $\mathcal L^{\otimes p}\simeq \mathcal O_{\mathfrak X}$ defines in turn an isomorphism $\left(L^{-1}\right)^{(p/\mathfrak X)} \simeq \mathbb G_{a,\mathfrak X}$. Thus the long exact sequence associated to the above short exact sequence identifies with :
\[ 0 \rightarrow t\cdot\ker(F:A\to A)      \rightarrow t\cdot A \xrightarrow{F} A \rightarrow \h^1(\mathfrak X, K)    \] 
where $t$ is the canonical section of $L^{-1}$ and $F$ is given by $ta \mapsto (ta)^p=sa^p$.

Let now $x$ be a closed point of $X$ in the support of $D$. Since $X$ is the moduli space of $\mathfrak X$, this also defines a unique closed point of $\mathfrak X$. All the considerations about torsors above apply again if one replaces $\mathfrak X$ by the residual gerbe $\mathcal G_x$ at $x$.

To conclude, let $a\in A$  such that $a$ vanishes at $x$ but $a\notin sA$ (this situation of course happens : take $A=k[u,v]$, $s=u$, $a=v$, $x=(0,0)$). Then $a$ defines a non trivial $K$-torsor over $\mathfrak X$, which remains non trivial over any fppf cover $A'\rightarrow A$. Thus the corresponding $G$-torsor is not isomorphic to the canonical lifting of  $Z^{-1}\rightarrow \mathfrak X$ over any fppf neighbourhood of $X$. But since $a$ vanishes at $x$, both $G$-torsors are isomorphic when restricted to $\mathcal G_x$.

\subsection{Second family}
\label{sub:second_family}

We conclude by a variant which shows that Proposition \ref{prop:alper_crit_torsors_local} also fails for Deligne-Mumford stacks of roots in odd characteristics.

The general setup is the same :  $X=\spec A$ is an affine scheme of positive characteristic $p\neq 2$, $s$ is a regular element of $A$, and $D$ is the corresponding Cartier divisor. Let now $\mathfrak X=\sqrt[2]{D/X}$, and as usual $\pi:\mathfrak X\to X$ be the canonical morphism. 
Again $\mathfrak X \simeq [Z/\mu_2]$, where $Z=\spec\left(\frac{ k[t]}{t^2-s}\right)$. We will use the corresponding morphism $\mathfrak X \xrightarrow{Z} \B\mu_2$ (this time we don't have to worry about the sign of $Z$ which is its own inverse).

Let $G=\mathbb \mu_p \rtimes \mathbb Z/2$, where $\mathbb Z/2$ acts non-trivially. Since $\mathbb Z/2\simeq \mu_2$ canonically, we can consider as before the gerbe $\mathcal G=\mathfrak X\times_{\B\mu_2} \B G$. Once again we have that $\mathcal G\simeq \B_{\mathfrak X}K$, where $K=Z\wedge^{\mathbb Z/2}\mu_p$.

 The group $K$ belongs to the twisted Kummer sequence :

 \[ 0\rightarrow K \rightarrow L^\times \xrightarrow{F}L^\times \rightarrow 0 \]
 where $L^\times=Z\wedge^{\mathbb Z/2}\mathbb G_m$ (since $\mathbb Z/2$ acts by group automorphisms, this $\mathbb G_m$-torsor is actually a $\mathfrak X$-group scheme).

 To compute the global sections, it is enough to notice that, as it happens for any ${\mathbb Z/2}$-torsor, the corresponding form $L^\times$ of $\mathbb G_m$ is the kernel of the norm morphism on the Weil restriction : \( L^\times= \ker( \R_{Z/\mathfrak X} \mathbb G_m \to \mathbb G_m )\). This provides us with the following description :

 \[ \Gamma(L^\times)=\left\{ a+bt, (a,b)\in A / a^2-sb^2=1\right\} \]  
 and the Frobenius is given by $a+bt \mapsto a^p+b^ps^{\frac{p-1}{2}}t$. Considering this, it is easy to use the cobord morphism $ \Gamma(L^\times) \to \h^1(\mathfrak X, K)$ to construct a non-trivial $K$-torsor on $\mathfrak X$ which is trivial on the residual gerbe on a given point $x$  of $D$. For instance assume that $b$ vanishes at $x$ but $b\notin sA$, and let $a$ be a square root of $1+sb^2$ (this occurs for instance when $A=k[[u,v]]$~: take $s=u$, $b=v$ and $x=(0,0)$). Then $a+bt$ defines the required torsor.

\section*{Acknowledgements}

The second author wishes to thank Eric Ahlqvist for correcting our definition of a weight, Philippe Gille for a precious piece of 
advice, Angelo Vistoli for numerous conversations on the subject, and David Rydh for his very enlightening comments on Proposition \ref{prop:alper_crit_torsors_local}, and his permission to include his counter-examples. Both authors are grateful to the referee, especially for sharing his views on Appendix \ref{sec:David_Rydh_counter_example}. This work was supported in part by the Labex CEMPI (ANR-11-LABX-0007-01) and by the International Centre for Theoretical Sciences (ICTS) during a visit for participating in the 
program -- Complex Geometry (Code: *ICTS/Prog-compgeo/2017/03*). The first author is 
partially supported by a J. C. Bose Fellowship.


\begin{thebibliography}{CEPT96}

\bibitem[AGV08]{agv:gw}
Dan Abramovich, Tom Graber and Angelo Vistoli,
\newblock Gromov-{W}itten theory of {D}eligne-{M}umford stacks.
\newblock {\em Amer. J. Math.}, 130(5):1337--1398, 2008.

\bibitem[Alp13]{alper:good_moduli}
Jarod Alper,
\newblock Good moduli spaces for {A}rtin stacks.
\newblock {\em Ann. Inst. Fourier (Grenoble)}, 63(6):2349--2402, 2013.

\bibitem[AHL20]{ahl:building}
Eric Ahlqvist,
\newblock Building data for stacky covers.
\newblock {\em Preprint}, 2020.

\bibitem[AOV08]{aov:tame_stacks}
Dan Abramovich, Martin Olsson and Angelo Vistoli,
\newblock Tame stacks in positive characteristic.
\newblock {\em Ann. Inst. Fourier (Grenoble)}, 58(4):1057--1091, 2008.

\bibitem[AV04]{arsvist:uniform}
Alessandro Arsie and Angelo Vistoli,
\newblock Stacks of cyclic covers of projective spaces.
\newblock {\em Compos. Math.}, 140(3):647--666, 2004.

\bibitem[BB17]{bisbor:fundamental_gerbe}
Indranil Biswas and Niels Borne,
\newblock The {N}ori fundamental {G}erbe of tame stacks.
\newblock {\em Transform. Groups}, 22(1):91--104, 2017.

\bibitem[BBN01]{bbn:par}
Vikraman Balaji, Indranil Biswas and Donihakkalu~S. Nagaraj,
\newblock Principal bundles over projective manifolds with parabolic structure over a divisor.
\newblock {\em Tohoku Math. J. (2)}, 53(3):337--367, 2001.

\bibitem[Bis97]{bis:par}
Indranil Biswas,
\newblock Parabolic bundles as orbifold bundles.
\newblock {\em Duke Math. J.}, 88(2):305--325, 1997.

\bibitem[Bor07]{bor:corr}
Niels Borne,
\newblock Fibr\'es paraboliques et champ des racines.
\newblock {\em Int. Math. Res. Not. IMRN}, (16):Art. ID rnm049, 38, 2007.

\bibitem[Bor09]{bor:rep}
Niels Borne,
\newblock Sur les repr\'esentations du groupe fondamental d'une vari\'et\'e
priv\'ee d'un diviseur \`a croisements normaux simples.
\newblock {\em Indiana Univ. Math. J.}, 58(1):137--180, 2009.

\bibitem[Bre94]{breen:tannaka}
Lawrence Breen,
\newblock Tannakian categories.
\newblock In {\em Motives ({S}eattle, {WA}, 1991)}, volume~55 of {\em Proc.
Sympos. Pure Math.}, pages 337--376. Amer. Math. Soc., Providence, RI, 1994.

\bibitem[BV12]{bv:par_sheaves}
Niels Borne and Angelo Vistoli,
\newblock Parabolic sheaves on logarithmic schemes.
\newblock {\em Adv. Math.}, 231(3-4):1327--1363, 2012.

\bibitem[BV15]{bv:nori_gerbe}
Niels Borne and Angelo Vistoli,
\newblock The {N}ori fundamental gerbe of a fibered category.
\newblock {\em J. Algebraic Geom.}, 24(2):311--353, 2015.

\bibitem[BV16]{bv:fundamental_gerbes}
Niels Borne and Angelo Vistoli,
\newblock Fundamental gerbes, 2016.

\bibitem[CEPT96]{CEPT:tame}
T.~Chinburg, B.~Erez, G.~Pappas and M.~J. Taylor,
\newblock Tame actions of group schemes: integrals and slices.
\newblock {\em Duke Math. J.}, 82(2):269--308, 1996.

\bibitem[Del89]{del:group_fond}
P.~Deligne,
\newblock Le groupe fondamental de la droite projective moins trois points.
\newblock In {\em Galois groups over {${\bf Q}$} ({B}erkeley, {CA}, 1987)},
volume~16 of {\em Math. Sci. Res. Inst. Publ.}, pages 79--297. Springer, New
York, 1989.

\bibitem[EV92]{esnvieh:vanishing}
H\'el\`ene Esnault and Eckart Viehweg,
\newblock {\em Lectures on vanishing theorems}, volume~20 of {\em DMV Seminar}.
\newblock Birkh\"auser Verlag, Basel, 1992.

\bibitem[GM71]{gm:tame}
Alexander Grothendieck and Jacob~P. Murre,
\newblock {\em The tame fundamental group of a formal neighbourhood of a
divisor with normal crossings on a scheme}.
\newblock Lecture Notes in Mathematics, Vol. 208. Springer-Verlag, Berlin-New
York, 1971.

\bibitem[Mar15]{marques:slices}
Sophie Marques,
\newblock Existence of slices on a tame context.
\newblock {\em Eur. J. Math.}, 1(1):54--77, 2015.

\bibitem[SP18]{stacks-project}
The Stacks Project Authors,
\newblock {\it The Stacks Project}.
\newblock 2018.

\bibitem[SR72]{rivano:cat_tan}
Neantro Saavedra~Rivano,
\newblock {\em Cat\'egories {T}annakiennes}.
\newblock Lecture Notes in Mathematics, Vol. 265. Springer-Verlag, Berlin-New
York, 1972.

\bibitem[TZ17]{tz:alg-nori-fund-gerbes}
Fabio Tonini and Lei Zhang,
\newblock Algebraic and Nori fundamental gerbes,
\newblock To appear in {\em Journal of the Institute of Mathematics of Jussieu}, 2017.

\bibitem[We38]{We} A. Weil, G\'en\'eralisation des
fonctions ab\'eliennes, \textit{Jour. Math. Pures Appl.}
\textbf{17} (1938), 47--87.

\bibitem[Wat79]{wat:aff}
William~C. Waterhouse,
\newblock {\em Introduction to affine group schemes}, volume~66 of {\em
Graduate Texts in Mathematics}.
\newblock Springer-Verlag, New York-Berlin, 1979.

\end{thebibliography}
\end{document}